\documentclass[11pt]{article}

\usepackage[a4paper,vmargin={25mm,25mm},hmargin={25mm,25mm}]{geometry}
\usepackage{amsmath,amssymb,amsthm,mathtools}
\usepackage{hyperref}
\usepackage{microtype}

\hypersetup{
  colorlinks=true,
  linkcolor=blue,
  citecolor=blue,
  urlcolor=blue
}

\numberwithin{equation}{section}

\newtheorem{theorem}{Theorem}[section]
\newtheorem{lemma}[theorem]{Lemma}
\newtheorem{proposition}[theorem]{Proposition}
\newtheorem{corollary}[theorem]{Corollary}

\newcommand{\R}{\mathbb R}
\newcommand{\eps}{\varepsilon}
\newcommand{\diam}{\operatorname{diam}}
\newcommand{\supp}{\operatorname{supp}}
\newcommand{\cN}{\mathcal N}
\newcommand{\cA}{\mathcal A}
\newcommand{\cF}{\mathcal F}
\newcommand{\cC}{\mathcal C}
\newcommand{\1}{\mathbf 1}

\title{Sharp Antipode--Neighbor Bounds in Every Dimension}
\author{Samuel Korsky}
\date{August 17, 2026}

\begin{document}
\maketitle

\begin{abstract}
\noindent
Let $P\subset\R^d$ be finite with $\diam(P)\le1$.  For $0<\eps<1$, let
$\cN_\eps(P)$ denote the number of ordered pairs at distance at most $\eps$,
including the diagonal, and let $\cA_\eps(P)$ denote the number of ordered
pairs at distance at least $1-\eps$.  We prove the sharp bound
\[
  \cN_\eps(P)\gtrsim_d \eps^{(d-1)/2}\cA_\eps(P)
\]
conjectured by Steinerberger (2025).  More generally, the same inequality holds
for finite positive Borel measures.  The main input is a uniform $L^2$
estimate, at cap thickness comparable to lattice spacing, for a
supporting-direction incidence operator.  In gnomonic coordinates, this
operator is controlled by an incidence estimate for small Fenchel--Young
gaps.
\end{abstract}

\section{Introduction}

Throughout, $\sigma$ denotes the usual, unnormalized surface measure on
$S^{d-1}$, and $\mathbb B_2^d$ denotes the closed Euclidean unit ball.  We
write $X\lesssim_{\alpha}Y$ if $X\le C_{\alpha}Y$ for a constant depending
only on the displayed parameters $\alpha$; the notations $\gtrsim_{\alpha}$
and $\asymp_{\alpha}$ are defined analogously.

Let $P\subset\R^d$ be finite with $\diam(P)\le1$.  For $0<\eps<1$, define
the ordered counts
\[
  \cN_\eps(P)
  :=\#\{(x,y)\in P^2:|x-y|\le\eps\},
  \qquad
  \cA_\eps(P)
  :=\#\{(x,y)\in P^2:|x-y|\ge1-\eps\}.
\]
Thus every diagonal pair $(x,x)$ is included in $\cN_\eps(P)$.  This
convention is essential for a cardinality-free statement valid for all
finite sets.

More generally, if $\mu$ is a finite positive Borel measure on $\R^d$, set
\begin{align*}
  \cN_\eps(\mu)
  &:=(\mu\times\mu)\{(x,y):|x-y|\le\eps\},\\
  \cA_\eps(\mu)
  &:=(\mu\times\mu)\{(x,y):|x-y|\ge1-\eps\}.
\end{align*}
For the counting measure $\mu_P:=\sum_{x\in P}\delta_x$, these are exactly
the preceding ordered counts.

Steinerberger asked how small the ratio
$\cN_\eps(P)/\cA_\eps(P)$ can be and proposed $\eps^{1/2}$ as the sharp
planar order \cite{Steinerberger2025}.  In earlier work, the author obtained
this order up to a logarithm in the plane and proved a higher-dimensional
estimate with exponent $3(d-1)/4$ \cite{Korsky2026}.  We prove the optimal
exponent in every fixed dimension.

\begin{theorem}[Weighted antipode--neighbor estimate]
\label{thm:weighted}
For every integer $d\ge1$, there is a constant $c_d>0$ such that, for every
finite positive Borel measure $\mu$ with compact support satisfying
$\diam(\supp\mu)\le1$ and every $0<\eps<1$,
\begin{equation}
  \cN_\eps(\mu)
  \ge
  c_d\eps^{(d-1)/2}\cA_\eps(\mu).
  \label{eq:weighted-main-bound}
\end{equation}
\end{theorem}

\begin{corollary}[Sharp antipode--neighbor estimate]
\label{cor:main}
For every integer $d\ge1$, there is a constant $c_d>0$ such that, for every
finite $P\subset\R^d$ with $\diam(P)\le1$ and every $0<\eps<1$,
\begin{equation}
  \cN_\eps(P)
  \ge
  c_d\eps^{(d-1)/2}\cA_\eps(P).
  \label{eq:main-bound}
\end{equation}
\end{corollary}

\subsection*{Sharpness}

For $d\ge2$, the exponent is best possible as $\eps\downarrow0$.  Let
$\mu$ be normalized surface measure on the sphere of radius $1/2$ in
$\R^d$.  If $\theta$ is the central angle between $x$ and $y$, then
\[
  |x-y|=\sin(\theta/2).
\]
Thus $|x-y|\le\eps$ means that $y$ lies in a spherical cap about $x$ of
angular radius $2\arcsin\eps\asymp\eps$, whereas
$|x-y|\ge1-\eps$ means that $y$ lies in a cap about $-x$ of angular radius
$2\arccos(1-\eps)\asymp\sqrt\eps$.  Standard spherical-cap estimates give
\begin{equation}
  \cN_\eps(\mu)\asymp_d\eps^{d-1},
  \qquad
  \cA_\eps(\mu)\asymp_d\eps^{(d-1)/2}.
  \label{eq:spherical-sharpness}
\end{equation}
Consequently, the power $(d-1)/2$ cannot be replaced by any smaller power.

The same scaling is attained by finite point sets.  Fix $\eps$ and let
$X_1,X_2,\ldots$ be independent samples with law $\mu$.  Almost surely the
points are distinct.  Moreover, the two distance-level sets
\[
  \{(x,y):|x-y|=\eps\},
  \qquad
  \{(x,y):|x-y|=1-\eps\}
\]
have $(\mu\times\mu)$-measure zero.  The strong law for $U$-statistics,
together with the vanishing diagonal contribution $1/n$, therefore gives,
for $P_n:=\{X_1,\ldots,X_n\}$,
\[
  \frac{\cN_\eps(P_n)}{n^2}\longrightarrow\cN_\eps(\mu),
  \qquad
  \frac{\cA_\eps(P_n)}{n^2}\longrightarrow\cA_\eps(\mu)
\]
almost surely.  Choosing $n=n(\eps)$ sufficiently large yields finite point
sets with the same asymptotic scaling as in
\eqref{eq:spherical-sharpness}.

For $d=1$, let $\nu$ be the normalized counting measure on
$\{-1/2,1/2\}$.  Then
\[
  \cN_\eps(\nu)=\cA_\eps(\nu)=\frac12
  \qquad(0<\eps<1),
\]
so the exponent $0$ is sharp in dimension one as well.  Steinerberger gives
further planar and higher-dimensional examples with the same scaling
\cite{Steinerberger2025}.

\subsection*{Idea of the proof}

The exponent comes from a quadratic angular loss.  Suppose that $z,w$ lie
in a compact convex set $B$, put $r=|z-w|$, and let
$v=(z-w)/r$.  If $\diam(B)\le\Delta$ and
$\theta=\angle(u,v)$, then
\[
  [h_B(u)-z\cdot u]+[h_B(-u)+w\cdot u]
  \le \Delta-r\cos\theta
  =(\Delta-r)+r(1-\cos\theta).
\]
For a near-diameter pair, $\Delta-r\lesssim\eps$, and
$1-\cos\theta\lesssim\theta^2$.  Hence every direction with
$\theta\lesssim\sqrt\eps$ is simultaneously an approximate supporting
direction for $z$ and the opposite of one for $w$.  A geodesic ball of
radius $\sqrt\eps$ in $S^{d-1}$ has measure
\[
  \asymp_d(\sqrt\eps)^{d-1}=\eps^{(d-1)/2}.
\]

To exploit this observation, discretize $\supp\mu$ by a grid of mesh
$\delta\asymp_d\eps$.  Let $m_z$ be the mass in the cell centered at $z$,
and let $B$ be the convex hull of the relevant centers.  For $z\in B$ and
$\tau>0$, set
\[
  \Omega_B(z,\tau)
  :=\{u\in S^{d-1}:h_B(u)-z\cdot u\le\tau\}
\]
and define the supporting-cap operator
\[
  \cC_{B,\tau}c(u)
  :=\sum_z c_z\1_{\Omega_B(z,\tau)}(u).
\]
The central analytic estimate, proved in Proposition
\ref{prop:cap-bound}, is the quantitative bound
\[
  \|\cC_{B,\tau}c\|_{L^2(S^{d-1})}
  \lesssim_d
  \left(1+\frac{\tau}{\delta}\right)^{d/2}
  \|c\|_{\ell^2}.
\]
Thus the operator norm is uniform over all compact convex sets whenever
cap thickness is comparable to lattice spacing.

With $F=\cC_{B,2\eps}m$, the proof reduces to the chain
\begin{align*}
  \cA_\eps(\mu)
  &\lesssim_d
    \eps^{-(d-1)/2}
    \int_{S^{d-1}}F(u)F(-u)\,d\sigma(u)\\
  &\le
    \eps^{-(d-1)/2}\|F\|_2^2\\
  &\lesssim_d
    \eps^{-(d-1)/2}\sum_zm_z^2\\
  &\le
    \eps^{-(d-1)/2}\cN_\eps(\mu).
\end{align*}

It remains to explain the cap estimate.  Away from a measure-zero set,
$S^{d-1}$ is covered by $2d$ open gnomonic charts.  In one chart, write a
direction as $(p,1)/\sqrt{1+|p|^2}$ with
$p\in(-1,1)^{d-1}$, set $Q=[-1,1]^{d-1}$, and write a lattice point
as $(x,t)$.  If $f(p)=H_B(p,1)$ and
$\beta(x)=\min_{q\in Q}(f(q)-q\cdot x)=-f_Q^*(x)$, then the homogeneous support
deficit has the exact decomposition
\[
  f(p)-p\cdot x-t
  =\underbrace{f(p)+f_Q^*(x)-p\cdot x}_{\text{Fenchel--Young gap}}
   +\underbrace{\beta(x)-t}_{\text{vertical slack}}.
\]
The lattice assumption controls the horizontal variables in the
Fenchel--Young estimate, while $\tau/\delta$ bounds the number of possible
vertical heights above each horizontal lattice point.

In one dimension, the relation
$\cF_f(s,a)\le\eta$ is a thickening of the monotone graph $\partial f$.
Its horizontal and vertical sections are intervals, and a cross-difference
identity gives a rectangle bound.  A weighted Schur argument then yields
the one-dimensional $L^2$ estimate.  Finally, the exact factorization
\[
  \cF_f((s,q),(a,b))
  =\cF_{\phi_q}(s,a)+\cF_{g_a}(q,b)
\]
allows the estimate to be iterated over the coordinates.  The logical
dependence is
\[
  \text{Lemma \ref{lem:biconvex}}
  \Longrightarrow
  \text{Proposition \ref{prop:one-dimensional}}
  \Longrightarrow
  \text{Theorem \ref{thm:fenchel-bound}}
  \Longrightarrow
  \text{Proposition \ref{prop:cap-bound}}
  \Longrightarrow
  \text{Theorem \ref{thm:weighted}}.
\]

\subsection*{Related work}

For all sufficiently small $\eps>0$, Steinerberger proved the planar lower
bound
\[
  \cN_\eps(P)
  \gtrsim
  \frac{\eps^{3/4}}{\log(1/\eps)^{1/4}}\cdot\cA_\eps(P)
\]
and identified $\eps^{1/2}$ as the natural candidate for the optimal order
\cite{Steinerberger2025}.  The author subsequently proved, again for all
sufficiently small $\eps$,
\[
  \cN_\eps(P)
  \gtrsim
  \frac{\eps^{1/2}}{\log(1/\eps)^{1/2}}\cdot\cA_\eps(P)
  \quad(d=2),
  \qquad
  \cN_\eps(P)
  \gtrsim_d
  \eps^{3(d-1)/4}\cA_\eps(P)
  \quad(d\ge3)
\]
\cite{Korsky2026}.  Those proofs reduce the problem to estimates for a
weighted graph on small boundary cells.  The present proof instead uses a
uniform $L^2$ operator bound for supporting directions and applies that
bound directly to the weighted pair count.

Dam\'asdi and Ploscaru independently obtained the sharp planar estimate 
and also proved a sharp two-parameter
planar theorem \cite{DamasdiPloscaru2026}. 
Their proof passes through a planar box graph and a degree-tail sparsity
estimate.  The present argument instead applies in every fixed dimension
and uses a uniform supporting-cap operator bound in place of planar
sparsity.

The exact diameter problem has a separate classical theory.  The diameter
graph of a finite set $P$ joins two distinct points precisely when their
distance equals $\diam(P)$.  A planar $n$-point set has at most $n$ edges
in this graph \cite{Erdos1946}; the sharp three-dimensional bound is
$2n-2$ \cite{Grunbaum1956,Heppes1956,Straszewicz1957}; and quadratic
examples exist in dimensions $d\ge4$ \cite{Swanepoel2009}.  These exact,
unweighted results do not directly control the approximate weighted cell
incidence considered here.

The convex-analytic component of the proof is based on the classical
Fenchel--Young inequality.  After supporting directions are parametrized by
Euclidean coordinates, the supporting-cap incidence condition reduces to
\[
  f(p)+f_Q^*(x)-p\cdot x\le\eta.
\]
Equivalently, $x$ is an $\eta$-subgradient of the extended-valued function
$f+\iota_Q$ at $p$, where $\iota_Q$ is $0$ on $Q$ and $+\infty$ outside
$Q$.  This equivalence is standard in Fenchel duality
\cite{Fenchel1949,BrondstedRockafellar1965,Rockafellar1970}.  A related
duality between Macbeath regions associated with a convex function and its
Legendre dual appears in \cite{AryaDaFonsecaMount2026}.

\section{A Biconvex Schur Estimate}\label{sec:biconvex}

Let $I\subset\R$ be a compact interval and let $\mathcal J$ be a finite
linearly ordered set.  Here and below, $|E|$ denotes one-dimensional
Lebesgue measure.  A subset of $\mathcal J$ is called an interval if
it contains every element lying between any two of its elements; the
empty set and singletons are allowed.  Consider a measurable kernel
$K:I\times\mathcal J\to\{0,1\}$.  For $a\in\mathcal J$ and $s\in I$,
write
\[
  I_a:=\{s\in I:K(s,a)=1\},
  \qquad
  \mathcal J_s:=\{a\in\mathcal J:K(s,a)=1\}.
\]
We call $K$ \emph{biconvex} if every $I_a$ is an interval in $I$ and
every $\mathcal J_s$ is an interval in the ordering of $\mathcal J$.

\begin{lemma}[Biconvex Schur estimate]\label{lem:biconvex}
Suppose that $K$ is biconvex and that there is $R>0$ such that
\begin{equation}
  |E|\,\#X\le R
  \label{eq:rectangle-condition}
\end{equation}
whenever $E\subset I$ is measurable, $X\subset\mathcal J$, and $K=1$
on $E\times X$.  Then
\[
  T_Kc(s):=\sum_{a\in\mathcal J}c_aK(s,a)
\]
satisfies
\begin{equation}
  \|T_Kc\|_{L^2(I)}
  \le4\sqrt R\,\|c\|_{\ell^2(\mathcal J)}.
  \label{eq:biconvex-norm}
\end{equation}
\end{lemma}

\begin{proof}
Set $d_a:=|I_a|$ and $n(s):=\#\mathcal J_s$, and discard columns with
$d_a=0$.

\emph{Column tails.}
Fix $a\in\mathcal J$ and an integer $k\ge1$, and put
$\ell=\lceil k/2\rceil$.  Let $B_a^-$ be the block of $\ell$
consecutive elements ending at $a$, when it exists, and let $B_a^+$ be
the block of $\ell$ consecutive elements beginning at $a$, when it
exists.  Any interval in $\mathcal J$ of cardinality at least $k$ that
contains $a$ has at least $\ell$ elements on one side of $a$, counting
$a$ itself, and therefore contains at least one of the two existing
blocks $B_a^-$ and $B_a^+$.  For either existing block $B$, the set
\[
  E_B:=\{s\in I:B\subset\mathcal J_s\}
\]
forms an all-one rectangle $E_B\times B$.  By
\eqref{eq:rectangle-condition},
$|E_B|\le R/\ell\le2R/k$.  Taking the union of the two possibilities and
also using the trivial bound by $d_a$, we obtain
\begin{equation}
  \big|\{s\in I_a:n(s)\ge k\}\big|
  \le
  \min\left(d_a,\frac{4R}{k}\right).
  \label{eq:column-tail}
\end{equation}

\emph{Row tails.}
Conversely, fix $s\in I$ and $r>0$.  For $a\in\mathcal J_s$, let
$L_a$ and $R_a$ be the lengths of the portions of $I_a$ to the left and
right of $s$.  Since a point has measure zero,
$L_a+R_a=d_a$.  Thus, if $d_a\ge r$, then either $L_a\ge r/2$ or
$R_a\ge r/2$.

Let $0<r'<r$ and set
\[
  E_s^-:=(s-r'/2,s),
  \qquad
  E_s^+:=(s,s+r'/2).
\]
If $L_a\ge r/2$, then $E_s^-\subset I_a$; in particular, this interval
lies inside the ambient interval $I$.  Similarly, if $R_a\ge r/2$, then
$E_s^+\subset I_a$.  Split the relevant columns according to these two
possibilities, assigning a column arbitrarily if both occur.  For each
nonempty class, the corresponding fixed interval, together with that class
of columns, is an all-one rectangle.  Since each fixed interval has length
$r'/2$, applying \eqref{eq:rectangle-condition} to the two classes gives
\[
  \#\{a\in\mathcal J_s:d_a\ge r\}
  \le
  \min\left(n(s),\frac{4R}{r'}\right).
\]
Letting $r'\uparrow r$ yields
\begin{equation}
  \#\{a\in\mathcal J_s:d_a\ge r\}
  \le
  \min\left(n(s),\frac{4R}{r}\right).
  \label{eq:row-tail}
\end{equation}

\emph{Layer-cake estimates.}
For $t>0$, take $k=\lfloor t^2\rfloor+1$ in
\eqref{eq:column-tail}.  The layer-cake formula gives
\begin{align}
  \int_{I_a}\sqrt{n(s)}\,ds
  &=\int_0^\infty
    \big|\{s\in I_a:\sqrt{n(s)}>t\}\big|\,dt \notag\\
  &\le\int_0^\infty
    \min\left(d_a,\frac{4R}{t^2}\right)dt
  =4\sqrt{Rd_a}.
  \label{eq:first-schur}
\end{align}
Similarly, using \eqref{eq:row-tail},
\begin{align}
  \sum_{a\in\mathcal J_s}\sqrt{d_a}
  &=\int_0^\infty
    \#\{a\in\mathcal J_s:\sqrt{d_a}>t\}\,dt \notag\\
  &\le\int_0^\infty
    \min\left(n(s),\frac{4R}{t^2}\right)dt
  =4\sqrt{Rn(s)}.
  \label{eq:second-schur}
\end{align}

\emph{Weighted Schur test.}
When $n(s)>0$, set
\[
  w(s,a):=\left(\frac{n(s)}{d_a}\right)^{1/4}
  \qquad(a\in\mathcal J_s).
\]
Weighted Cauchy--Schwarz gives
\[
  |T_Kc(s)|^2
  \le
  \left(\sum_{a\in\mathcal J_s}|c_a|^2w(s,a)^2\right)
  \left(\sum_{a\in\mathcal J_s}w(s,a)^{-2}\right)
\]
\[
  =
  \left(\sum_{a\in\mathcal J_s}|c_a|^2
    \frac{\sqrt{n(s)}}{\sqrt{d_a}}\right)
  \left(\sum_{a\in\mathcal J_s}
    \frac{\sqrt{d_a}}{\sqrt{n(s)}}\right).
\]
By \eqref{eq:second-schur}, the second factor is at most $4\sqrt R$.
Integrating, interchanging the sum and integral, and using
\eqref{eq:first-schur}, we obtain
\[
\begin{aligned}
  \|T_Kc\|_{L^2(I)}^2
  &\le4\sqrt R
    \sum_{a\in\mathcal J}|c_a|^2\frac{1}{\sqrt{d_a}}
    \int_{I_a}\sqrt{n(s)}\,ds\\
  &\le16R\sum_{a\in\mathcal J}|c_a|^2.
\end{aligned}
\]
This proves \eqref{eq:biconvex-norm}.
\end{proof}

\section{Fenchel--Young Incidence Bounds}\label{sec:fenchel}

Let $Q\subset\R^m$ be compact and convex, and let $f:Q\to\R$ be
continuous and convex.  Its conjugate restricted to $Q$ is
\[
  f_Q^*(x):=\max_{p\in Q}\bigl(p\cdot x-f(p)\bigr),
  \qquad x\in\R^m,
\]
and the corresponding Fenchel--Young gap is
\begin{equation}
  \cF_f(p,x):=f(p)+f_Q^*(x)-p\cdot x\ge0.
  \label{eq:fenchel-gap}
\end{equation}
The gap is zero exactly when $p$ realizes the maximum in the definition
of $f_Q^*(x)$.  If $\widetilde f:=f+\iota_Q$ denotes the
extended-valued convex function on $\R^m$, then
$\cF_f(p,x)\le\eta$ is equivalent to
$x\in\partial_\eta\widetilde f(p)$.

In one dimension, the relation $\cF_f(s,a)\le\eta$ is a thickening of
the monotone graph $\partial f$.  Its horizontal and vertical sections are
therefore intervals.

\begin{proposition}[One-dimensional Fenchel--Young bound]
\label{prop:one-dimensional}
Let $f:[-1,1]\to\R$ be continuous and convex, let $\delta>0$, let
$\Lambda\subset\delta\mathbb Z$ be finite, and let $\eta\ge0$.  Then
\[
  T_fc(s):=\sum_{a\in\Lambda}c_a
  \1_{\{\cF_f(s,a)\le\eta\}}
\]
satisfies
\begin{equation}
  \|T_fc\|_{L^2([-1,1])}
  \le
  C\left(1+\frac{\eta}{\delta}\right)^{1/2}
  \|c\|_{\ell^2(\Lambda)},
  \label{eq:one-dimensional-bound}
\end{equation}
where $C$ is an absolute constant.
\end{proposition}

\begin{proof}
For fixed $a$, the set
\[
  J_a:=\{s\in[-1,1]:\cF_f(s,a)\le\eta\}
\]
is an interval, since $s\mapsto\cF_f(s,a)$ is convex.  For fixed $s$,
the admissible $a$ form an interval in the ordering of $\Lambda$, since
$a\mapsto f_{[-1,1]}^*(a)-as$ is convex.  Thus the incidence kernel is
biconvex.

Suppose that the kernel is one on $E\times X$, and put $k=\#X$.  The
cases $E=\varnothing$ and $k=0$ are trivial, and if $k=1$, then
$k|E|\le2$.  If $k\ge2$, let
$a_-:=\min X$ and $a_+:=\max X$.  For $s,t\in E$, direct cancellation
gives the exact identity
\begin{equation}
  \cF_f(s,a_-)+\cF_f(t,a_+)
  -\cF_f(s,a_+)-\cF_f(t,a_-)
  =(s-t)(a_+-a_-).
  \label{eq:cross-difference}
\end{equation}
All four gaps on the left lie in $[0,\eta]$, so
\[
  |s-t|(a_+-a_-)\le2\eta.
\]
The crucial point is that the cross-difference eliminates $f$ completely;
thus the resulting rectangle estimate is uniform over all convex
functions.
Taking the supremum over $s,t\in E$, using $|E|\le\diam(E)$ and
$a_+-a_-\ge(k-1)\delta$, we find
\[
  k|E|
  \le\frac{2k}{k-1}\cdot\frac{\eta}{\delta}
  \le\frac{4\eta}{\delta}.
\]
Thus the rectangle hypothesis of Lemma \ref{lem:biconvex} holds with
\[
  R_0:=\max\left(2,\frac{4\eta}{\delta}\right)
  \le4\left(1+\frac{\eta}{\delta}\right).
\]
Lemma \ref{lem:biconvex} gives
\[
  \|T_fc\|_2
  \le8\left(1+\frac{\eta}{\delta}\right)^{1/2}\|c\|_2,
\]
which proves \eqref{eq:one-dimensional-bound}.
\end{proof}

\begin{theorem}[Multidimensional Fenchel--Young bound]
\label{thm:fenchel-bound}
Let $m\ge1$, let $Q=[-1,1]^m$, let $f:Q\to\R$ be continuous and convex,
let $\delta>0$, let $\Lambda\subset\delta\mathbb Z^m$ be finite, and
let $\eta\ge0$.  Define
\[
  T_fc(p):=\sum_{x\in\Lambda}c_x
  \1_{\{\cF_f(p,x)\le\eta\}}.
\]
Then
\begin{equation}
  \|T_f\|_{\ell^2(\Lambda)\to L^2(Q)}
  \le
  C_m\left(1+\frac{\eta}{\delta}\right)^{m/2},
  \label{eq:fenchel-bound}
\end{equation}
where $C_m$ depends only on $m$.  In particular, if
$0\le\eta\le A\delta$, then the operator norm is bounded by a constant
depending only on $m$ and $A$.
\end{theorem}

\begin{proof}
We induct on $m$.  Proposition \ref{prop:one-dimensional} is the case
$m=1$.

Write $p=(s,q)$ and $x=(a,b)$, where $s,a\in\R$ and
$q,b\in\R^{m-1}$, and put $Q':=[-1,1]^{m-1}$.  For each first
coordinate $a$ occurring in $\Lambda$, define
\[
  g_a(q):=\min_{s\in[-1,1]}\bigl(f(s,q)-as\bigr).
\]
The function $g_a$ is convex.  Indeed, for $q_0,q_1\in Q'$ choose
minimizers $s_0,s_1$, and for $0\le\lambda\le1$ set
$q_\lambda=\lambda q_0+(1-\lambda)q_1$ and
$s_\lambda=\lambda s_0+(1-\lambda)s_1$.  Joint convexity gives
\[
  g_a(q_\lambda)
  \le f(s_\lambda,q_\lambda)-as_\lambda
  \le\lambda g_a(q_0)+(1-\lambda)g_a(q_1).
\]
Moreover,
\[
  |g_a(q)-g_a(q')|
  \le\sup_{s\in[-1,1]}|f(s,q)-f(s,q')|,
\]
so uniform continuity of $f$ on the compact cube shows that $g_a$ is
continuous.

For fixed $q$, let $\phi_q(s):=f(s,q)$.  The relevant conjugates satisfy
\begin{align*}
  g_{a,Q'}^*(b)
  &=\max_{q\in Q'}\max_{s\in[-1,1]}
    \bigl(as+b\cdot q-f(s,q)\bigr)
   =f_Q^*(a,b),\\
  \phi_{q,[-1,1]}^*(a)
  &=\max_{s\in[-1,1]}\bigl(as-f(s,q)\bigr)
   =-g_a(q).
\end{align*}
Consequently, the gap factors exactly as
\begin{equation}
  \cF_f((s,q),(a,b))
  =\cF_{\phi_q}(s,a)+\cF_{g_a}(q,b).
  \label{eq:gap-factorization}
\end{equation}
Both terms are nonnegative.  This exact identity is the mechanism that
permits coordinate iteration.  Hence
\begin{equation}
  \1_{\{\cF_f\le\eta\}}
  \le
  \1_{\{\cF_{\phi_q}\le\eta\}}
  \1_{\{\cF_{g_a}\le\eta\}}.
  \label{eq:indicator-factorization}
\end{equation}

Because the kernel is nonnegative, $|T_fc|\le T_f|c|$.  Set
\[
  u_a(q)
  :=\sum_{b:(a,b)\in\Lambda}|c_{a,b}|
  \1_{\{\cF_{g_a}(q,b)\le\eta\}}.
\]
By \eqref{eq:indicator-factorization},
\[
  |T_fc(s,q)|
  \le
  \sum_a u_a(q)
  \1_{\{\cF_{\phi_q}(s,a)\le\eta\}}.
\]
Apply Proposition \ref{prop:one-dimensional} in the $s$ variable for
each fixed $q$, integrate over $q$, and apply the induction hypothesis
to each $g_a$ on the lattice fiber
$\{b:(a,b)\in\Lambda\}\subset\delta\mathbb Z^{m-1}$.  This gives
\[
\begin{aligned}
  \|T_fc\|_{L^2(Q)}^2
  &\le
  C_1^2\left(1+\frac{\eta}{\delta}\right)
  \sum_a\|u_a\|_{L^2(Q')}^2\\
  &\le
  C_1^2C_{m-1}^2
  \left(1+\frac{\eta}{\delta}\right)^m
  \sum_{(a,b)\in\Lambda}|c_{a,b}|^2.
\end{aligned}
\]
Taking $C_m=C_1C_{m-1}$ proves \eqref{eq:fenchel-bound}.
\end{proof}

\section{Supporting-Direction Caps}\label{sec:caps}

For a compact convex set $K\subset\R^d$, let
\[
  H_K(q):=\max_{y\in K}q\cdot y,
  \qquad q\in\R^d,
\]
be its homogeneous support function, and write
$h_K:=H_K|_{S^{d-1}}$.  For $z\in K$ and $\tau>0$, define
\[
  \Omega_K(z,\tau)
  :=\{u\in S^{d-1}:h_K(u)-z\cdot u\le\tau\}.
\]
Thus $\Omega_K(z,\tau)$ consists of directions in which $z$ has support
deficit at most $\tau$.

\begin{proposition}[Quantitative supporting-cap $L^2$ bound]
\label{prop:cap-bound}
Let $d\ge2$, let $K\subset\R^d$ be compact and convex, let $\delta>0$,
let $\Gamma\subset\delta\mathbb Z^d\cap K$ be finite, and let $\tau>0$.
Define
\[
  \cC_{K,\tau}c(u)
  :=\sum_{z\in\Gamma}c_z\1_{\Omega_K(z,\tau)}(u).
\]
Then
\begin{equation}
  \|\cC_{K,\tau}c\|_{L^2(S^{d-1},d\sigma)}
  \le
  C_d\left(1+\frac{\tau}{\delta}\right)^{d/2}
  \|c\|_{\ell^2(\Gamma)},
  \label{eq:cap-bound}
\end{equation}
where $C_d$ depends only on $d$.
\end{proposition}

The lattice assumption controls the horizontal variables in the
Fenchel--Young estimate, while the ratio $\tau/\delta$ controls the number
of possible vertical heights over each horizontal lattice point.

\begin{proof}
Put $m=d-1$.  Apart from the measure-zero set on which two coordinates
have equal maximal absolute value, the sphere is the disjoint union of the
$2d$ open regions on which one signed coordinate is strictly larger in
absolute value than all the others.  On each region, a signed coordinate
permutation $S$ sends the distinguished coordinate to the positive last
coordinate, and every direction in the region has a unique representation
\begin{equation}
  Su=\frac{(p,1)}{\sqrt{1+|p|^2}},
  \qquad p\in(-1,1)^m.
  \label{eq:gnomonic-chart}
\end{equation}
We enlarge the parameter domain to $Q:=[-1,1]^m$ when applying Theorem
\ref{thm:fenchel-bound}.  In these coordinates,
\[
  d\sigma(u)=(1+|p|^2)^{-d/2}\,dp,
\]
so spherical and Lebesgue measure are comparable on the chart.

Because $S$ is a signed permutation, it preserves
$\delta\mathbb Z^d$.  Write $Sz=(x,t)$ with
$x\in\delta\mathbb Z^m$ and $t\in\delta\mathbb Z$.  Let $K_S:=SK$ and
define
\[
  f(p):=H_{K_S}(p,1),
  \qquad p\in Q.
\]
For $u$ and $p$ related by \eqref{eq:gnomonic-chart}, define the
homogeneous support deficit
\begin{equation}
  \operatorname{def}_z(p)
  :=f(p)-p\cdot x-t
  =\sqrt{1+|p|^2}\,[h_K(u)-z\cdot u].
  \label{eq:homogeneous-deficit}
\end{equation}

Define the chartwise support barrier
\[
  \beta(x):=\min_{p\in Q}\bigl(f(p)-p\cdot x\bigr)=-f_Q^*(x).
\]
Because the slopes are restricted to the chart $Q$, no assertion is made
that $\beta(x)$ equals the actual upper boundary of the vertical fiber of
$K_S$.  What is needed is the support inequality: if
$(x,t)=Sz\in K_S$, then $f(p)\ge p\cdot x+t$ for every $p\in Q$, and hence
$t\le\beta(x)$.

The support deficit decomposes exactly as
\begin{equation}
  \underbrace{\operatorname{def}_z(p)}_{\text{homogeneous support deficit}}
  =
  \underbrace{\cF_f(p,x)}_{\text{Fenchel--Young gap}}
  +
  \underbrace{\beta(x)-t}_{\text{vertical slack}}.
  \label{eq:deficit-decomposition}
\end{equation}
Indeed, $f_Q^*(x)=-\beta(x)$.  Both terms on the right are nonnegative, so
\begin{equation}
  \cF_f(p,x)\le\operatorname{def}_z(p),
  \qquad
  0\le\beta(x)-t\le\operatorname{def}_z(p).
  \label{eq:gap-below-deficit}
\end{equation}
Geometrically, the first term records how far the horizontal coordinate
$x$ is from being dual to the chart direction $p$, while the second records
how far the height $t$ lies below the chartwise supporting barrier.

For each horizontal lattice point $x$, define
\[
  \mathcal T_x
  :=\{t\in\delta\mathbb Z:
      S^{-1}(x,t)\in\Gamma
      \text{ and }
      \Omega_K(S^{-1}(x,t),\tau)
      \text{ meets the current chart}\}.
\]
If $t\in\mathcal T_x$, then \eqref{eq:homogeneous-deficit} is at most
$\sqrt d\,\tau$ at some $p\in(-1,1)^m$.  By
\eqref{eq:gap-below-deficit},
\[
  t\le\beta(x)\le t+\sqrt d\,\tau.
\]
Thus every $t\in\mathcal T_x$ lies in an interval of length
$\sqrt d\,\tau$, and therefore
\begin{equation}
  \#\mathcal T_x
  \le
  \left\lfloor\frac{\sqrt d\,\tau}{\delta}\right\rfloor+1
  \le C_d\left(1+\frac{\tau}{\delta}\right).
  \label{eq:vertical-multiplicity}
\end{equation}

For each $x$, aggregate the coefficients over the relevant heights:
\[
  \widetilde c_x
  :=\sum_{t\in\mathcal T_x}|c_{S^{-1}(x,t)}|.
\]
By Cauchy--Schwarz and \eqref{eq:vertical-multiplicity},
\begin{equation}
  \sum_x|\widetilde c_x|^2
  \le
  C_d\left(1+\frac{\tau}{\delta}\right)
  \sum_{z\in\Gamma}|c_z|^2.
  \label{eq:aggregation-bound}
\end{equation}
Moreover, \eqref{eq:homogeneous-deficit} and
\eqref{eq:gap-below-deficit} give, pointwise on the chart,
\[
  |\cC_{K,\tau}c(u(p))|
  \le
  \sum_x\widetilde c_x
  \1_{\{\cF_f(p,x)\le\sqrt d\,\tau\}}.
\]
Theorem \ref{thm:fenchel-bound}, applied with
$\eta=\sqrt d\,\tau$, yields
\[
  \|\cC_{K,\tau}c\|_{L^2(\text{current chart})}
  \le
  C_d\left(1+\frac{\tau}{\delta}\right)^{(d-1)/2}
  \|\widetilde c\|_{\ell^2}.
\]
Combining this with \eqref{eq:aggregation-bound} gives
\[
  \|\cC_{K,\tau}c\|_{L^2(\text{current chart})}
  \le
  C_d\left(1+\frac{\tau}{\delta}\right)^{d/2}
  \|c\|_{\ell^2(\Gamma)}.
\]
Summing the squared estimates over the $2d$ chart pieces proves
\eqref{eq:cap-bound}.
\end{proof}

\section{Discretization and Proof of the Main Theorem}
\label{sec:proof-main}

For a compact convex set $B\subset\R^d$, define its width in direction
$u\in S^{d-1}$ by
\[
  w_B(u):=h_B(u)+h_B(-u).
\]
Since two supporting points in the directions $u$ and $-u$ are at
distance at most $\diam(B)$,
\begin{equation}
  w_B(u)\le\diam(B)
  \qquad(u\in S^{d-1}).
  \label{eq:width-below-diameter}
\end{equation}
For $\Omega\subset S^{d-1}$, write
$-\Omega:=\{-u:u\in\Omega\}$.

\begin{lemma}[A common geodesic ball of supporting directions]
\label{lem:common-directions}
Let $B\subset\R^d$ be compact and convex with
$\diam(B)\le\Delta$, and let $z,w\in B$ with $r:=|z-w|>0$.  Put
$v:=(z-w)/r$.  Then, for every $u\in S^{d-1}$,
\begin{equation}
\begin{aligned}
 &[h_B(u)-z\cdot u]+[h_B(-u)+w\cdot u]\\
 &\qquad
 =w_B(u)-r\cos\angle(u,v)
 \le\Delta-r\cos\angle(u,v).
\end{aligned}
\label{eq:deficit-width-inequality}
\end{equation}
In particular, if $\Delta-r\le\kappa$, $0<\rho\le1$, and
\begin{equation}
  \kappa+\frac{\Delta\rho^2}{2}\le\tau,
  \label{eq:common-direction-condition}
\end{equation}
then the geodesic ball of radius $\rho$ about $v$ is contained in
\[
  \Omega_B(z,\tau)\cap\bigl(-\Omega_B(w,\tau)\bigr).
\]
\end{lemma}

\begin{proof}
The identity in \eqref{eq:deficit-width-inequality} follows by direct
cancellation, and the inequality follows from
\eqref{eq:width-below-diameter}.  Both support deficits on the left are
nonnegative.  If $\angle(u,v)\le\rho$, then
\[
  \Delta-r\cos\angle(u,v)
  =\Delta-r+r(1-\cos\angle(u,v))
  \le\kappa+\frac{\Delta\rho^2}{2},
\]
where we used $r\le\diam(B)\le\Delta$ and
$1-\cos t\le t^2/2$.  Under
\eqref{eq:common-direction-condition}, each nonnegative deficit is at
most $\tau$.
\end{proof}

\begin{proof}[Proof of Theorem \ref{thm:weighted}]
If $\mu=0$, the result is trivial.  Both sides of
\eqref{eq:weighted-main-bound} are quadratic in $\mu$, so by homogeneity we
may assume
\[
  \mu(\R^d)=1.
\]

We first dispose of $1/10\le\eps<1$ in every dimension.  Choose
$x_0\in\supp\mu$.  Since
$\supp\mu\subset x_0+\mathbb B_2^d$, cover this ball by
$M_d=O_d(1)$ sets of diameter at most $1/10$, and refine the cover to a
measurable partition of $\supp\mu$.  If $m_1,\ldots,m_{M_d}$ are the
masses of the parts, then
\[
  \cN_\eps(\mu)
  \ge\sum_{j=1}^{M_d}m_j^2
  \ge\frac1{M_d},
  \qquad
  \cA_\eps(\mu)\le1.
\]
Since $\eps^{(d-1)/2}\le1$, the theorem follows in this range after
adjusting $c_d$.

Assume henceforth that $0<\eps<1/10$.  We next prove the one-dimensional
case.  Let
$a:=\min\supp\mu$ and $b:=\max\supp\mu$, and put
\[
  L:=[a,a+\eps],
  \qquad
  R:=[b-\eps,b],
  \qquad
  \ell:=\mu(L),
  \qquad
  r:=\mu(R).
\]
If $x\le y$ and $y-x\ge1-\eps$, then, because $b-a\le1$,
\[
  x\le b-(1-\eps)\le a+\eps,
  \qquad
  y\ge a+(1-\eps)\ge b-\eps.
\]
Thus every near-diameter ordered pair lies in $L\times R$ or
$R\times L$, and
\[
  \cA_\eps(\mu)\le2\ell r.
\]
Every pair in $L\times L$ and every pair in $R\times R$ is an
$\eps$-neighbor, so
\[
  \cN_\eps(\mu)
  \ge\max(\ell^2,r^2)
  \ge\ell r
  \ge\frac12\cA_\eps(\mu).
\]
This proves the theorem for $d=1$.

Now let $d\ge2$.  If $\cA_\eps(\mu)=0$, there is nothing to prove.
Otherwise
\begin{equation}
  D:=\diam(\supp\mu)\in[1-\eps,1].
  \label{eq:diameter-range}
\end{equation}
Set
\[
  \delta:=\frac{\eps}{10\sqrt d},
  \qquad
  r_0:=\frac{\sqrt d\,\delta}{2}=\frac{\eps}{20},
\]
and partition $\R^d$ into the half-open cubes
\[
  C_z:=z+[-\delta/2,\delta/2)^d,
  \qquad z\in\delta\mathbb Z^d.
\]
Let $\Gamma$ be the finite set of centers $z$ for which
$C_z\cap\supp\mu\ne\varnothing$, and put $m_z:=\mu(C_z)$.  Some of these
masses may be zero.  Each cube has diameter
$\sqrt d\,\delta=\eps/10$, so every ordered pair in the same cube is an
$\eps$-neighbor.  Hence
\begin{equation}
  \cN_\eps(\mu)
  \ge\sum_{z\in\Gamma}m_z^2.
  \label{eq:neighbor-cell-bound}
\end{equation}

Set
\[
  B:=\operatorname{conv}(\Gamma),
  \qquad
  \Delta:=D+2r_0=D+\frac{\eps}{10}.
\]
Every center in $\Gamma$ is within $r_0$ of $\supp\mu$.  Hence, for any
$z,w\in\Gamma$, one can choose $x,y\in\supp\mu$ with
$|z-x|,|w-y|\le r_0$, and therefore
\[
  |z-w|\le|x-y|+2r_0\le D+2r_0=\Delta.
\]
Taking a convex hull does not increase diameter, so
\begin{equation}
  \diam(B)\le\Delta.
  \label{eq:convex-hull-diameter}
\end{equation}

Write $z\sim w$ if there are
$x\in\supp\mu\cap C_z$ and $y\in\supp\mu\cap C_w$ with
$|x-y|\ge1-\eps$; sums over $z\sim w$ are over ordered cell pairs.  Every
near-diameter pair belongs to one such cell product, so
\begin{equation}
  \cA_\eps(\mu)
  \le\sum_{z\sim w}m_zm_w.
  \label{eq:antipode-cell-bound}
\end{equation}

Fix $z\sim w$ and witnesses $x,y$.  Then
\[
  r:=|z-w|
  \ge1-\eps-2r_0
  =1-\frac{11}{10}\eps.
\]
Consequently,
\[
  \Delta-r
  \le D-1+\eps+4r_0
  \le\frac65\eps,
  \qquad
  \Delta\le1+\frac{\eps}{10}<\frac{11}{10}.
\]
Apply Lemma \ref{lem:common-directions} with
\[
  \kappa=\frac65\eps,
  \qquad
  \rho=\sqrt\eps,
  \qquad
  \tau=2\eps.
\]
Indeed,
\[
  \kappa+\frac{\Delta\rho^2}{2}
  <\left(\frac65+\frac{11}{20}\right)\eps
  =\frac74\eps
  <2\eps.
\]
Thus, if $v=(z-w)/|z-w|$, the geodesic ball of radius $\sqrt\eps$ about
$v$ lies in
\[
  \Omega_z\cap(-\Omega_w),
  \qquad
  \Omega_z:=\Omega_B(z,2\eps).
\]
Writing $B_{\mathrm{geo}}(v,\rho)$ for the geodesic ball of radius
$\rho$ about $v$, we have, for $0<\rho\le1$,
\[
  \sigma(B_{\mathrm{geo}}(v,\rho))
  =\sigma(S^{d-2})\int_0^\rho\sin^{d-2}t\,dt
  \gtrsim_d\rho^{d-1}.
\]
Taking $\rho=\sqrt\eps$ gives
\begin{equation}
  1
  \le
  C_d\eps^{-(d-1)/2}
  \int_{S^{d-1}}
  \1_{\Omega_z}(u)\1_{\Omega_w}(-u)\,d\sigma(u)
  \qquad(z\sim w).
  \label{eq:entrywise-incidence}
\end{equation}

Define
\[
  F(u):=\sum_{z\in\Gamma}m_z\1_{\Omega_z}(u)
       =\cC_{B,2\eps}m(u).
\]
Summing \eqref{eq:entrywise-incidence} with weights $m_zm_w$, enlarging
the sum from the relation $z\sim w$ to all ordered cell pairs, and using
Cauchy--Schwarz and the invariance of $\sigma$ under $u\mapsto-u$, we
obtain
\begin{align*}
  \cA_\eps(\mu)
  &\le\sum_{z\sim w}m_zm_w\\
  &\le C_d\eps^{-(d-1)/2}
    \int_{S^{d-1}}F(u)F(-u)\,d\sigma(u)\\
  &\le C_d\eps^{-(d-1)/2}\|F\|_2^2.
\end{align*}
Proposition \ref{prop:cap-bound} applies with $K=B$, $\tau=2\eps$, and
\[
  \frac{\tau}{\delta}=20\sqrt d.
\]
Therefore
\[
  \|F\|_2^2
  \le C_d\sum_{z\in\Gamma}m_z^2.
\]
Combining this with \eqref{eq:neighbor-cell-bound} yields
\[
  \cA_\eps(\mu)
  \le C_d\eps^{-(d-1)/2}\cN_\eps(\mu),
\]
which is equivalent to \eqref{eq:weighted-main-bound}.
\end{proof}

\begin{proof}[Proof of Corollary \ref{cor:main}]
Apply Theorem \ref{thm:weighted} to the counting measure
$\mu_P=\sum_{x\in P}\delta_x$.
\end{proof}

\section*{Acknowledgments}

The operator-norm formulation was developed with assistance from GPT-5.6
Pro, which was also used to explore, draft, and check several technical
arguments.  The author takes full responsibility for the final text, all
mathematical claims, and any errors.

\end{document}